\documentclass{article}
\usepackage[all,ps,arc]{xy}  
\usepackage{amsmath,amssymb,latexsym}
\usepackage{amsfonts}
\usepackage{amsthm}
\newcommand{\point}{\ensuremath{\xymatrix{A\ar@<+.6ex>[r]^(.5){\alpha}
&B\ar@<+.6ex>[l]^(.5){\beta}}}}
\newcommand{\rg}{\ensuremath{\xymatrix{A\ar@<+1ex>[r]^{\alpha}\ar@<-1ex>[r]_{\gamma}&B\ar[l]|{\beta}}}}

\newtheorem{Theorem}{Theorem}[section]
\newtheorem{Lemma}[Theorem]{Lemma}
\newtheorem{Proposition}[Theorem]{Proposition}
\newtheorem{Definition}[Theorem]{Definition}

\newtheorem{Corollary}[Theorem]{Corollary}

\theoremstyle{remark}\newtheorem{nsc}{}[section]
\theoremstyle{remark}\newtheorem{Remark}[Theorem]{\bf Remark}
\theoremstyle{remark}\newtheorem{Example}[Theorem]{\bf Example}

\newcommand{\bA}{{\mathbb A}}
\newcommand{\bB}{{\mathbb B}}

\newcommand{\bN}{{\mathbb N}}

\newcommand{\bX}{{\mathbb X}}

\newcommand{\cA}{{\mathcal A}}
\newcommand{\cB}{{\mathcal B}}
\newcommand{\cC}{{\mathcal C}}

\newcommand{\cF}{{\mathcal F}}
\newcommand{\cE}{{\mathcal E}}

\newcommand{\cK}{{\mathcal K}}

\newcommand\Pt{\mathsf{Pt}}
\newcommand\RGph{\mathsf{RGph}}
\newcommand\RRel{\mathsf{RRel}}
\newcommand\EqRel{\mathsf{EqRel}}
\newcommand\EffRel{\mathsf{EffRel}}
\newcommand\Set{\mathsf{Set}}
\newcommand\Gpd{\mathsf{Gpd}}
\newcommand\Gp{\mathsf{Gp}}
\newcommand\Top{\mathsf{Top}}

\newcommand\K{\mathsf{K}}
\newcommand\coker{\mathsf{coker}}
\newcommand\Rel{\mathrm{Rel}}
\newcommand\Nor{\mathrm{Nor}}

\begin{document}

\def \tm{\!\times\!}

\newenvironment{changemargin}[2]{\begin{list}{}{
\setlength{\topsep}{0pt}
\setlength{\leftmargin}{0pt}
\setlength{\rightmargin}{0pt}
\setlength{\listparindent}{\parindent}
\setlength{\itemindent}{\parindent}
\setlength{\parsep}{0pt plus 1pt}
\addtolength{\leftmargin}{#1}\addtolength{\rightmargin}{#2}
}\item}{\end{list}}

\title{Bourn-normal monomorphisms\\
in regular Mal'tsev categories}
\author{Giuseppe Metere}

\maketitle

\begin{abstract}
Normal monomorphisms in the sense of Bourn describe the equivalence classes of an internal equivalence relation. Although the definition is given in the fairly general setting of a category with finite limits, later investigations on this subject often focus on protomodular settings, where normality becomes a property.
This paper clarifies the connections between internal equivalence relations and Bourn-normal monomorphisms in regular Mal'tesv categories with pushouts of split monomorphisms along arbitrary morphisms, whereas a full description is achieved for quasi-pointed regular Mal'tsev categories with pushouts of split monomorphisms along arbitrary morphisms.
\end{abstract}

\section{Introduction}
In \cite{Bo00b}, Bourn introduces a notion of normal monomorphism that translates in categorical terms the set-theoretical notion of \emph{equivalence class} of a given equivalence relation. Although Bourn gives his definition with minimal assumptions, the main theme of the cited article is developed in the pointed protomodular context, where equivalence relations are completely determined by their zero classes.
Indeed, for protomodular categories, the author shows that when a subobject is Bourn-normal to an equivalence relation, then the equivalence relation is essentially unique. However, he does not provide any procedure in order to determine such a relation. Borceux highlighted the problem in \cite{Borceux04}, where he writes: \emph{unfortunately, there is no known general construction of this equivalence relation, given a subobject ``candidate to be normal''}.

The present paper originates from the aim to clarify these issues in a context possibly broader than that of protomodular categories.

Indeed, our recipe for the construction of the equivalence relation associated with a given Bourn-normal monomorphism (Proposition \ref{prop:Rel(n)}) relies on a more general result concerning the induced bifibration on an epi-reflective subcategory of a given bifibred category (see Corollary \ref{coroll:bifib_reflection}). Actually, for a Bourn-normal monomorphism $n$, an  equivalence relation $\Rel(n)$  to which $n$ is Bourn-normal is obtained using an opcartesian lift of $n$ with respect to the functor that sends an equivalence relation on $X$, to the base  object $X$. In fact, when considered as an internal functor, this lift happens to be also cartesian, and a discrete opfibration.

We will show that the same fact holds in a regular Mal'tsev category, where, in general,  the equivalence relation to which $n$ is Bourn-normal is not unique. Nonetheless, $\Rel(n)$ is still special, in some sense: it is the smallest equivalence relation in the lattice of all equivalence relations that have $n$ as one of their classes. Finally, from the construction of the functor $\Rel$, we get a characterization of Bourn-normal monomorphisms that allows us to remove the existential quantifier from Definition \ref{def:normal} (see Corollary \ref{cor:no_exists}).

A brief description of the contents of this paper follows: Section 2 is devoted to setting up the necessary preliminaries; in Section 3, after recalling the definition of Bourn-normal monomorphism, Bourn's normalization is extended to the quasi-pointed setting; Section 4 focuses on the more general problem of the induced bifibration on an epi-reflective subcategory, while, in Section 5, this is applied to reflexive graphs and relations in order to associate an equivalence relation with a given Bourn-normal monomorphism;  Section 6 gathers the results of Section 3 and of Section 5: in the quasi-pointed Mal'tsev regular setting, Theorem \ref{thm:The_Theorem}  gives a general description of the relationship between Bourn-normal monomorphisms and internal equivalence relations; the last Section presents some algebraic examples.

\section{Preliminaries}\label{sec:preliminaries}
In this section, we recall some basic concepts from \cite{BB}, and fix notation. The reader may consult \cite{Borceux94.2} for the definition of fibration, opfibration and related notions (in [loc.\ cit.]  \emph{op}fibrations are called \emph{co}fibrations). Throughout the section we shall assume that $\cC$ is a category with finite limits.

\begin{nsc}
We denote by $\Pt(\cC)$ the category with objects the four-tuples $(B,A,b,s)$ in $\cC$, with $b\colon B\to A$ and $b\cdot s=1_A$, and with morphisms $(f,g)\colon (D,C,d,s)\to (B,A,b,s)$:
\begin{equation}\label{diag:split_epi_morphism}
\begin{aligned}\xymatrix{
D\ar@<-.5ex>[d]_{d}\ar[r]^{f}
&B\ar@<-.5ex>[d]_{b}
\\
C\ar@<-.5ex>[u]_{s}\ar[r]_g
&A\ar@<-.5ex>[u]_{s}}
\end{aligned}
\end{equation}
such that both the upward and the downward directed squares commute.
\end{nsc}
\begin{nsc}
The assignment  $(B,A,b,s)\mapsto A$ gives rise to a fibration, the so called \emph{fibration of points}:
$$
\cF\colon \Pt(\cC)\to\cC.
$$
For an object $A$ of $\cC$, we denote by $\Pt_A(\cC)$ the fibre of $\cF$ over $A$.
$\cF$-cartesian morphisms are morphisms $(f,g)$ in $\Pt(\cC)$  such that the downward directed square of diagram (\ref{diag:split_epi_morphism}) is a pullback.
In this way, any morphism $g\colon C\to A$ defines a ``change of base'' functor $g^*\colon \Pt_A(\cC)\to\Pt_C(\cC)$.

Since the category $\cC$ is  finitely complete, also the fibres $\Pt_A(\cC)$ are, and every change of base functor is left exact.
\end{nsc}
\begin{nsc}
The category $\cC$ is called \emph{protomodular} when every change of base of the fibration of points is conservative, i.e.\ when they reflect  isomorphisms (see \cite{BB}).
\end{nsc}
\begin{nsc}
When  $\cC$  admits an initial object $0$, for any object $A$ of $\cC$, one can consider the change of base along the initial arrow $!_A\colon 0\to A$. This defines a \emph{kernel functor} $\cK_A$, for every object $A$. In the presence of an initial object, the protomodularity condition can be equivalently defined by requiring that just kernel functors be conservative.
\end{nsc}
\begin{nsc}
In presence of initial and final objects, a category $\cC$ is called quasi-pointed when the unique arrow $0\to 1$ is a
monomorphism. If this is the case, the domain functor $\Pt_0(\cC)\to \cC$  defines an embedding. Its image is the subcategory $\cC_0$ spanned by objects with null support (i.e.\ objects $A$ equipped with a necessarily unique arrow $\omega_A\colon A\to 0$), and we have a factorization:
$$
\cK_A\colon \Pt_A(\cC)\to\cC_0\hookrightarrow\cC\,.
$$
When $0\to 1$ is an isomorphism, we say that $\cC$ is pointed; if this is the case, clearly $\cC_0=\cC$.
\end{nsc}
\begin{nsc}
Let $\cC$ be quasi-pointed. We shall call \emph{normal} any $f\colon X\to Y$ that appears as a pullback of an initial arrow. In other words, $f$ is normal if it fits into a pullback diagram as it is shown below:
\begin{equation}\label{diag:kernel}
\begin{aligned}
\xymatrix{
X\ar[r]^{\omega_X}\ar[d]_{f}
&0\ar[d]^{!_Z}
\\
Y\ar[r]_{g}
&Z
}
\end{aligned}
\end{equation}
In this case, we write $f=\mathsf{ker}(g)$, meaning that $f$ is the \emph{kernel} of $g$. We denote by $\K(\cC)$ the class of normal monomorphisms  of a given category $\cC$. The notation adopted comes from the fact that, in the pointed case, normal monomorphism are just kernels of some morphism.
\end{nsc}
\begin{nsc}
Following \cite{Bo91}, we say that $g$ is the cokernel of $f$, and we write $g=\coker(f)$, when $(\ref{diag:kernel})$ is a pushout. Note that this definition of cokernel is \emph{not} dual to that of kernel, unless the category $\cC$ is pointed. 
We recall from \cite{Bo01} that, if $\cC$ is quasi-pointed and protomodular, every regular epimorphism is the cokernel of its kernel.
\end{nsc}
\begin{nsc}
A reflexive graph $\bA=(A_1,A_0,d,c,e)$ in $\cC$ is a diagram
$\xymatrix{A_1\ar@<1ex>[r]^{d}\ar@<-1ex>[r]_{c}&A_0\ar[l]|{e}}$ with $d\cdot e=1_{A_0} =c\cdot e$.
A morphism of reflexive graphs
$$
\xymatrix{F \colon \bA\ar[r]&\bB=(B_1,B_0,d,c,e)}
$$
is a pair of maps $f_i\colon A_i\to B_i$ ($i=0,1$)
\begin{equation}\label{diag:RGph_morphism}
\begin{aligned}
\xymatrix{
A_1\ar@<-1ex>[d]_{d}\ar@<+1ex>[d]^{c}\ar[r]^{f_1}
&B_1\ar@<-1ex>[d]_{d}\ar@<+1ex>[d]^{c}
\\
A_0\ar[u]|e\ar[r]_{f_0}
&B_0\ar[u]|e
}
\end{aligned}
\end{equation}
such that
$d\cdot f_1 = f_0\cdot d$,
$c\cdot f_1 = f_0\cdot c$ and
$f_1\cdot e = e\cdot f_0$.
Reflexive graphs in $\cC$ and their morphisms form a category we shall denote by $\RGph(\cC)$.
\end{nsc}
\begin{nsc}
Recall that a reflexive relation $R=(R,r,X)$ on a given object $X$ is a subobject $(R,r)$ of
$X\times X$ that contains the diagonal of $X$:
$$
\xymatrix{
&R\ar[d]^{r}
\\
X \ar[r]_-{\langle1,1\rangle}\ar[ur]^{e}
&X\times X
}
$$
If $(S,s,Y)$ is another reflexive relation, on the object $Y$, a morphism of relations $R\to S$ is an
arrow $f\colon X\to Y$ such that $f\times f$ restricts to $R\to S$.

Reflexive relations in $\cC$ and their morphisms form a category we shall denote by $\RRel(\cC)$.
\end{nsc}

\begin{nsc}\label{nsc:embedding_RRel-RGph}
There is an obvious full replete embedding
$$
\xymatrix{
I\colon \RRel(\cC)\ar[r]&\RGph(\cC)
}
$$
defined by $I(R,r,X)=(R,X,p_1\cdot r,p_2\cdot r,e)$ , where $p_1$ and $p_2$ are the product projections.

From now on, we shall freely identify reflexive relations with their associated reflexive graphs.
Consequently, we shall denote the relation $R=(R,r,X)$ by $(R,r_1,r_2,e,X)$,  $(R,r_1,r_2,X)$  or just
$(R,r_1,r_2)$, where $r_i= p_i\cdot r$, for $i=1,2$.
\end{nsc}

\begin{nsc}\label{qui}
As a consequence of \ref{nsc:embedding_RRel-RGph}, a morphism of reflexive relations
$$\xymatrix{f\colon (R,r_1,r_2,X) \ar[r]&(S,s_1,s_2,Y)}$$
will be usually described by a diagram:
$$
\xymatrix{
R\ar@<-1ex>[d]_{r_1}\ar@<+1ex>[d]^{r_2}\ar[r]^{f_|}
&S\ar@<-1ex>[d]_{s_1}\ar@<+1ex>[d]^{s_2}
\\
X\ar[u]\ar[r]_{f}
&Y\ar[u]
}
$$
It is termed \emph{discrete fibration} when the square $s_2\cdot f_|=f\cdot r_2$ is a pullback, \emph{discrete opfibration} when the square $s_1\cdot f_|=f\cdot r_1$ is a pullback. Let us observe that the commutativity of the upward directed diagram is a consequence of the  commutativity of the two downward directed ones.
\end{nsc}

\begin{nsc}\label{nsc:embedding_EqRel-RRel}
Recall that a reflexive relation $(R,r_1,r_2,X)$ is an equivalence relation if it is also symmetric and transitive. Equivalence relations form a full subcategory  $\EqRel(\cC)$ of $\RRel(\cC)$. A morphism of equivalence relations is a discrete fibration if, and only if, it is a discrete opfibration.

A finitely complete category $\cC$ is termed \emph{Mal'tsev} when $\EqRel(\cC)=\RRel(\cC)$. All protomodular categories are Mal'tsev (see \cite{BB}).
\end{nsc}

\begin{nsc}
Recall that an (internal) equivalence relation is called \emph{effective} when it is the kernel pair of a map. A category $\cC$ is \emph{regular}, if it is finitely
complete, it has pullback-stable regular epimorphisms, and all effective
equivalence relations admit coequalizers.
A regular category $\cC$ is \emph{Barr-exact} when all equivalence relations are
effective (see \cite{Ba71}).
\end{nsc}

\begin{nsc}
Quasi-pointed protomodular regular categories are called
\emph{sequentiable}. If they are pointed, they are called \emph{homological}, and they are termed \emph{semi-abelian} when they are also Barr exact and have finite coproducts (see \cite{BB}).
\end{nsc}

\section{Bourn-normal monomorphisms in quasi-pointed categories}
In \cite{Bo00b}, Bourn introduces a notion of normality that identifies those subobjects that should be considered as the classes of a given internal equivalence relation. In other words, he translates the set-theoretical notion of equivalence class in categorical terms.

Then, for a pointed category $\cC$, he presents a procedure to determine a Bourn-normal monomorphism canonically associated with a given equivalence relation $R$.
The same method applies if the base category is just quasi-pointed. We analyze  here some relevant constructions whose proofs are substantially the same as those in [loc.\ cit.]. Bourn's results are recovered when the base category is pointed.

\begin{Definition}[\cite{Bo00b}]
In a category $\cC$ with finite limits, a morphism $n\colon N\to X$ is \emph{Bourn-normal} to an equivalence relation $(R,r_1,r_2)$ on the object $X$ when $n\times n$ factors via $\langle r_1,r_2\rangle$ and the following two commutative diagrams are pullbacks:
\begin{equation}\label{diag:normal}
\begin{aligned}
\xymatrix{\ar@{}[dr]|(.3){\lrcorner}
N\times N\ar[r]^-{\tilde{n}}\ar@{=}[d]
&R\ar[d]^-{\langle r_1,r_2\rangle}
\\
N\times N \ar[r]_(.45){n\times n}
&X \times X
}
\qquad
\xymatrix{\ar@{}[dr]|(.3){\lrcorner}
N\times N\ar[d]_-{p_1}\ar[r]^-{\tilde{n}}
&R\ar[d]^-{r_1}
\\
N\ar[r]_-{n}
&X
}
\end{aligned}
\end{equation}
\end{Definition}
\begin{nsc}
According to what  stated in \ref{qui}, the definition of  Bourn-normal morphism can be rephrased by saying that the morphism of equivalence relations
$$
\xymatrix{
N\times N\ar@<-.6ex>[d]_-{p_1}\ar@<+.6ex>[d]^-{p_2}\ar[r]^-{\tilde{n}}
&R\ar@<-.6ex>[d]_-{r_1}\ar@<+.6ex>[d]^-{r_2}
\\
N\ar[r]_-{n}
&X
}
$$
is a discrete fibration.

\end{nsc}
\begin{nsc}
Bourn has shown that all Bourn-normal maps are monomorphic. Moreover, when the category $\cC$ is protomodular, if $n$ is Bourn-normal to a relation $R$, then $R$ is essentially unique. In other words, normality becomes a property, (see \cite{Bo00b}).
\end{nsc}

Even if the category $\cC$ is not protomodular, still it is possible to define a notion of Bourn-normal monomorphism as a property. However, in this case, we cannot avoid  using the existential quantifier in the definition.
\begin{Definition}\label{def:normal}
We call a morphism $n$ \emph{Bourn-normal} when there exists an equivalence relation $R$ such that $n$ is Bourn-normal to $R$.
\end{Definition}

The class of Bourn-normal monomorphisms  inherits an obvious categorical structure, from the category of arrows, $\mathsf{Arr}(\cC)$. We shall denote the category of Bourn-normal monomorphisms by $\mathsf{N}(\cC)$.

\subsection{A Bourn-normal monomorphism associated with an equivalence relation}

\begin{nsc}\label{nsc:kernels_normal_to_effective}
Let us recall that in a quasi-pointed finitely complete category $\cC$, the kernel of a map $f$ is Bourn-normal to the effective equivalence relation given by the kernel pair of $f$.
A converse of this assertion holds if $\cC$ is pointed, i.e.\ if a map $n$ is Bourn-normal to an effective equivalence relation, then $n$ is normal. Generally, this is not true in the quasi-pointed case, however, a weaker form of it holds true (see Proposition \ref{prop:ker}). Let us begin our analysis from not-necessarily effective equivalence relations.
\end{nsc}

\begin{Proposition}\label{prop:rel_to_norm}
In a quasi-pointed finitely complete category $\cC$, we are given an equivalence relation $(R,r_1,r_2)$ on the object $X$.
Then, if $k=\ker(r_1)$, the map $r_2\cdot k$ is Bourn-normal to $R$.
\end{Proposition}
The composition  $r_2\cdot k$ is known as the \emph{normalization} of the (equivalence) relation $(R,r_1,r_2)$.
\begin{proof}
Let us consider the following diagram:
\begin{equation}\label{diag:rel_to_normal}
\begin{aligned}
\xymatrix{
K\times K\ar[r]^-{\bar{k}}\ar@<-.5ex>[d]_{p_1}\ar@<+.5ex>[d]^{p_2}
&R\times_X R\ar[r]^-{m}\ar@<-.5ex>[d]_{p_1}\ar@<+.5ex>[d]^{p_2}
&R\ar@<-.5ex>[d]_{r_1}\ar@<+.5ex>[d]^{r_2}
\\
K\ar[r]_-{k}\ar[d]
&R\ar[r]_-{r_2}\ar[d]^{r_1}
&X
\\
0\ar[r]
&X
}
\end{aligned}
\end{equation}
where $(K,k)$ is the kernel of $r_1$, $(R\times_X R,p_1,p_2)$ is the kernel pair of $r_1$ and $m$ is the \emph{twisted-transitivity}  map that (set theoretically) associates the element $yRz$ with the pair $(xRy, xRz)$. It is well known that in such a situation, the two squares on the right (whose vertical arrows have the same index) are pullbacks. Moreover, as recalled in \ref{nsc:kernels_normal_to_effective}, $k$ is Bourn-normal to $(R\times_X R,p_1,p_2)$, so that the upper left squares  (whose vertical arrows have the same index) are pullbacks too. By pasting the upper squares, one easily deduces that the map $r_2\cdot k$ is Bourn-normal to $R$.
\end{proof}

The process described above gives rise to a functor.
\begin{Proposition}\label{prop:Nor(R)}
Let $\cC$ be a finitely complete quasi-pointed category. The construction described in Proposition \ref{prop:rel_to_norm}  extends to a faithful functor
$$
\xymatrix{
\Nor\colon \EqRel(\cC)\ar[r]&\mathsf{N}(\cC)
}
$$
that associates with each equivalence relation $(R,r_1,r_2)$, the monomorphism
$r_2\cdot k$, which is Bourn-normal to $R$.
\end{Proposition}

\begin{nsc}
The Bourn-normal subobjects with which $\Nor$ associates a relation are indeed rather special ones, since their domains
have null support. We shall denote by $\mathsf{N}_0(\cC)$ the full subcategory of Bourn-normal monomorphisms whose domains have null support.
It will be clearer later (as a consequence of Corollary \ref{cor:norm_0-norm}) that $\mathsf{N}_0(\cC)$ is precisely the essential image of the functor $\Nor$.
\end{nsc}

Finally, we can prove a sort of converse of what is stated in \ref{nsc:kernels_normal_to_effective}.

\begin{Proposition}\label{prop:ker}
Let $\cC$ be a quasi-pointed finitely complete category.
If the morphism $\xymatrix{n\colon N\ar[r]&X}$ is Bourn-normal to the \emph{effective} equivalence relation
$$(X\times_Y X,p_1,p_2),$$  kernel pair of $\xymatrix{f\colon X\ar[r]&Y}$, then
$\ker(f)\leq n$ (as subobjects of $X$). The equality holds if, and only if,
$n$ is in $\mathsf{N}_0(\cC)$.
\end{Proposition}

\begin{proof}
Let us consider the kernel $(K,k)$ of $\xymatrix{N\ar[r]&1}$ with its associated kernel pair relation (see diagram (\ref{diag:norm_to_0-norm})) together  with $n$ and its associated kernel pair relation:
\begin{equation}
\begin{aligned}
\xymatrix{
K\times K\ar[r]^-{k\times k}\ar@<-.5ex>[d]_{p_1}\ar@<+.5ex>[d]^{p_2}
&N\times N\ar@<-.5ex>[d]_{p_1}\ar@<+.5ex>[d]^{p_2}\ar[r]^{n_|}
&X\times_Y X\ar@<-.5ex>[d]_{p_1}\ar@<+.5ex>[d]^{p_2}
\\
K\ar[r]_-{k}\ar[d]
&N\ar[r]_{n}
&X\ar[d]^{f}
\\
0\ar[rr]
&&Y
}
\end{aligned}
\end{equation}
As the left vertical fork is not only left but also right exact, the initial arrow produces the lower commutative rectangle. Moreover, since (upper) pullbacks compose by Corollary 2 in \cite{Bo00b}, the lower rectangle is a pullback, i.e.\ $n\cdot k=\ker(f)$.
Finally, the comparison between $\ker(f)$ and $n$ is the monomorphism $k$, so that
$\ker(f)\leq n$ as desired. The last statement is obvious.
\end{proof}

The next proposition, and the following corollary, will be useful later. They provide an alternative way to compute the  normalization of an equivalence relation $R$ when a Bourn-normal subobject of $R$ is already known.

\begin{Proposition}\label{prop:norm_0-norm}
Let $\cC$ be a quasi-pointed finitely complete category.
If the morphism $\xymatrix{n\colon N\ar[r]&X}$ is Bourn-normal to an  equivalence relation
$(R,r_1,r_2)$, then
$$
\Nor(R)=n\cdot k,
$$
where $k=\ker(\xymatrix{N\ar[r]&1})$.
\end{Proposition}
\begin{proof}
Let $k'=\ker(r_1)$. We know from the definition that $\Nor(R)=r_2\cdot k'$.
Let us consider the following pasting of pullback diagrams (of solid arrows):
$$
\xymatrix{
K\ar[d]\ar[r]^-{\langle0,1\rangle}
&K\times K\ar@<-.5ex>[d]_{p_1}\ar@<+.5ex>@{-->}[d]^{p_2}\ar[r]^-{k_|}
&N\times N\ar@<-.5ex>[d]_{p_1}\ar@<+.5ex>@{-->}[d]^{p_2}\ar[r]^-{n_|}
&R\ar@<-.5ex>[d]_{r_1}\ar@<+.5ex>@{-->}[d]^{r_2}
\\
0\ar[r]
&K\ar[r]_-{k}
&N\ar[r]_-{n}
&X
}
$$
It shows that $k'=n_|\cdot k_| \cdot\langle0,1\rangle$.
Then, using the dashed arrows, one computes
$r_2\cdot k'=r_2\cdot n_|\cdot k_| \cdot\langle0,1\rangle=n\cdot k\cdot p_2 \cdot\langle0,1\rangle=n\cdot k$.
\end{proof}

Let us observe that, in the proof, we were allowed to use the arrow $\langle0,1\rangle$ in the first pullback on the left, even for $\cC$ only quasi-pointed, for the reason that $K$ has null support.
The pasting of the two pullbacks on the left shows that, since the object $N$ does not have null support, the kernel of the product projections from $N\times N$ is not $N$, as in the pointed cases, but $K$, i.e.\ the kernel of the final map $N\to 1$.

The following statement follows immediately from Proposition \ref{prop:norm_0-norm}, and it extends Proposition \ref{prop:ker} to not necessarily effective equivalence relations, so that where equivalence relations determine monomorphisms with null support, effective relations determine normal monomorphisms.

\begin{Corollary}\label{cor:norm_0-norm}
Let $\cC$ be a quasi-pointed finitely complete category.
If $n$ is Bourn-normal to an equivalence relation $R$, then $\Nor(R)\leq n$. The equality holds if and only if
$n$ is in $\mathsf{N}_0(\cC)$.
\end{Corollary}

Note that we get a factorization $\Nor=L\cdot \Nor_0$ through the embedding $$\xymatrix{L\colon\mathsf{N}_0(\cC)\ \ar@{>->}[r]&\mathsf{N}(\cC)}.$$

A relevant consequence of Corollary \ref{cor:norm_0-norm} is that there is an essentially unique Bourn-normal monomorphism in $\mathsf{N}_0(\cC)$, which is Bourn-normal to a given equivalence relation.

\section{The induced bifibration on an epi-reflective subcategory}
Let $V$ be the functor
$$
\xymatrix{\EqRel(\cC)\ar[r]&\cC}
$$
that maps an equivalence relation on an object $X$, to the object $X$ itself. It is well known that, if $\cC$ has pullbacks, then $V$ is a fibration, with cartesian maps given by fully faithful morphisms of equivalence relations.

Then, the pullback on the left of diagram (\ref{diag:normal}) in the definition of Bourn's normality, can be interpreted as a cartesian lift of $n$ with respect to the fibration $V$.

The definition of Bourn-normal monomorphism can be restated as follows: a morphism $\xymatrix{N\ar[r]^n&X}$ is Bourn-normal to the equivalence relation $(R,r_1,r_2)$ on $X$ if the following two conditions are fulfilled:
\begin{itemize}
\item $n^*(R)=\nabla N$
\item the cartesian lift of $n$, $\xymatrix{\nabla N \ar[r]^-{\tilde{n}}& R}$, is a discrete opfibration.
\end{itemize}
This formulation, suggests to investigate the possible (op)fibrational aspects involved in the definition of Bourn-normal monomorphism. In the present section, we take a more formal approach by studying when a (op)fibration restricts to a reflective subcategory. We will then apply our results in order to show that the functor $V$ is a bifibration.

\begin{nsc} Notation.
Only for the rest of this section, we shall adopt a slightly different notation  for categories, functors, objects of a category etc. In this way, we mean to stress the formal approach undertaken.
\end{nsc}

\begin{nsc}
A \emph{bifibration} is a functor $F\colon \cA \to \cC$ that is at the same time a fibration and an opfibration \cite{Gr59}. In this section we discuss when a bifibration $F$ as above, induces  by restriction a bifibration $FI$ on an epi-reflective subcategory $I\colon\cB\hookrightarrow\cA$. We start with a basic lemma.
 \end{nsc}
\begin{Lemma}\label{lemma:e_vertical}
Let us consider the following diagram of categories and functors:
\begin{equation}\label{diag:basic_sit}
\begin{aligned}
\xymatrix{
\cA \ar@/^2ex/[rr]^{R}  \ar@<+1ex>@{}[rr]|{\bot}\ar[dr]_{F}
&&
\cB \ar[ll]^{I}\ar[dl]^{FI}
\\
&\cC}\end{aligned}
\end{equation}
where $(R\dashv I, \eta,\epsilon)$ is an adjunction. Then, if $\eta$ has $F$-vertical components (i.e.\ $F(\eta_a)=id, \forall a\in\cA$),  also $\epsilon$ has $FI$-vertical components.
\end{Lemma}
\begin{proof}
Apply $F$ to the triangle identity
$$
\xymatrix{
Ib\ar[r]^-{\eta_{Ib}}\ar[dr]_{id}
&IRIb\ar[d]^{I\epsilon_b}
\\&Ib
}
$$
and use $F(\eta_{Ib})=id$.
\end{proof}
In the case of the lemma, the adjunction restricts to fibres.

\begin{Proposition}\label{prop:fib_reflection}
Let us consider diagram (\ref{diag:basic_sit}), where $(R\dashv I, \eta,\epsilon)$ is a full replete  epi-reflection,  $F$ is a fibration and $\eta$ has $F$-vertical components. Then the following hold:
\begin{itemize}
\item[(i)] every $F$-cartesian map with its codomain in $\cB$ is itself in $\cB$;
\item[(ii)] $FI$ is a fibration with the same cartesian lift as $F$;
\item[(iii)] $I$ is a cartesian functor over $\cC$.
\end{itemize}
\end{Proposition}
\begin{proof}[Proof of (i)]
Let us consider a cartesian arrow
$$
\xymatrix{
a\ar[r]^-{\bar{\kappa}}&Ib\,.
}
$$
What we aim to prove is that there exists a (unique) arrow $\kappa$ in $\cB$ such that $I\kappa=\bar{\kappa}$.
By the universal property of the unit of the adjunction, there exits a unique arrow
$\xymatrix{Ra\ar[r]^-{\tau}&b}$ in $\cB$, such that the following diagram commutes:
$$
\xymatrix{
a\ar[r]^-{\bar{\kappa}}\ar[d]_{\eta_a}
&Ib
\\
IRa\ar@{-->}[ur]_-{I\tau}
}
$$

Since $\eta_a$ is $F$-vertical, not only $\bar{\kappa}$, but also $I\tau$ is a lift of $F\bar{\kappa}$, hence $\bar{\kappa}$ cartesian implies that there exists a unique
$\xymatrix{IRa\ar[r]^-{\alpha}&a}$ such that $\bar{\kappa}\cdot \alpha = I\tau$. Composing, one gets
$$
\bar{\kappa}\cdot \alpha \cdot \eta_a= I\tau\cdot \eta_a =\bar{\kappa}\,.
$$
Now, $\bar{\kappa}$ is cartesian, so that uniqueness implies $\alpha\cdot \eta_a=id$, i.e.\ $\eta_a$ is a split-monomorphism. On the other hand, by hypothesis, it is also an epimorphism, and therefore an isomorphism. The result now follows from the fact that $I$ is replete and fully faithful.
\end{proof}

\begin{proof}[Proof of (ii)]
Given the pair $(b, \xymatrix{c\ar[r]^-{\gamma}&FIb})$, we are to find a $FI$-cartesian lift of $\gamma$ relative to $b$.
Since $F$ is a fibration, we actually have a $F$-cartesian lift of $\gamma$  relative to $Ib$, say an arrow $\xymatrix{a\ar[r]^-{\bar\kappa}&Ib}$.
By point (i) above this lies in $\cB$, i.e.\ there exists a unique $\kappa$ in $\cB$ such that $I\kappa=\bar\kappa$. We claim that such a $\kappa$ is $FI$-cartesian.
In order to prove this assertion, let us consider an arrow $\beta$ of $\cB$ with codomain $b$ and an arrow $\gamma'$ of $\cC$ such that $FI\beta=\gamma\cdot\gamma'$:
$$
\xymatrix@C=10ex{
b''\ar@/^3ex/[drr]^-{\beta}\ar@{-->}[dr]
\\
&b'\ar[r]^{\kappa}
&b\ar@{.}[d]
&\cB\ar[d]^{FI}
\\
c''\ar[r]_{\gamma'}
&c' \ar[r]_{\gamma }
&FIb
&\cC
}
$$
We apply the functor $I$ to $\beta$ and $\kappa$ above, and since $\bar\kappa$ is $F$-cartesian, there exists a unique arrow $$\xymatrix{Ib''\ar[r]^{\alpha'}&Ib'}$$ such that $F\alpha'=\gamma'$ and $I\kappa\cdot\alpha' =I\beta$.
Then, since $I$ is full and faithful, there exists a unique $\beta'$ such that ($I\beta'=\alpha'$, so that) $FI\beta'=\gamma'$ and $\kappa\cdot\beta'=\beta$.
\end{proof}

\begin{proof}[Proof of (iii)]
Given a $FI$-cartesian arrow $\xymatrix{b'\ar[r]^{\kappa}&b}$, we are to prove that $I\kappa$ is $F$-cartesian. Let $\bar{\kappa}$ be a $F$ cartesian lift of $FI\kappa$ relative to $Ib$ (it exists since $F$ is a fibration). Since $I\kappa$ is a lift of $FI\kappa$  relative to $Ib$,  we have a unique vertical arrow $\alpha$ such that $\bar{\kappa}\cdot \alpha=I\kappa$:
$$
\xymatrix@C=10ex{
Ib'\ar@{-->}[d]_{\alpha}\ar@/^2ex/[dr]^{I\kappa}
\\
\bar{a}\ar[r]^{\bar{\kappa}}
&Ib
&\cA\ar[d]^F
\\
FIb'\ar[r]_{FI\kappa}
&FIb
&\cC}
$$
By (i) above, there exists a unique arrow $\lambda$ in $\cB$ such that $I\lambda=\bar{\kappa}$. Furthermore, since $I$ is full and faithful, there exists a unique arrow $\beta$ in $\cB$ such that $I\beta=\alpha$. As a consequence, the commutative triangle in the above diagram is the image via $I$ of a commutative triangle in $\cB$:
$$
\xymatrix@C=10ex{
b'\ar[d]_{\beta}\ar@/^2ex/[dr]^{\kappa}
\\
\bar{b}\ar[r]^{\lambda}
&b
&\cB\ar[d]^{FI}
\\
FIb'\ar[r]_{FI\kappa}
&FIb
&\cC}
$$
Now, $\kappa$ is $FI$-cartesian by hypothesis, so that there exists a unique $FI$-vertical arrow $\beta'$ such that $\kappa \cdot \beta'=\lambda$. This in turns implies that $\beta'\cdot\beta =id_{b'}$. On the other hand, $I\lambda=\bar{\kappa}$ $F$-cartesian implies $I\beta\cdot I\beta' =id_{\bar{b}}$. Hence, since $I$ is full and faithful, one easily deduces that $\beta$ is a vertical isomorphism, so that $I\kappa$ is $F$-cartesian.
\end{proof}
Now we let opfibrations enter in the picture.
\begin{Proposition}\label{prop:opfib_reflection}
Let us consider diagram (\ref{diag:basic_sit}), where $(R\dashv I, \eta,\epsilon)$ is a full replete  epi-reflection,  $F$ is an opfibration and $\eta$ has $F$-vertical components. Then the following hold:
\begin{itemize}
\item[(iv)] $FI$ is an opfibration with  $FI$-opcartesian lifts obtained by reflecting $F$-opcartesian lifts;
\item[(v)] $R$ is an opcartesian functor over $\cC$.
\end{itemize}
\end{Proposition}
\begin{proof}[Proof of (iv)]
Given the pair $(b, \xymatrix{FIb\ar[r]^-{\gamma}&c})$, we are to find a $FI$-opcartesian lift of $\gamma$  relative to $b$.
Since $F$ is an opfibration, we actually have a $F$-opcartesian lift of $\gamma$ relative  to $Ib$, say an arrow $\xymatrix{Ib\ar[r]^-{\bar\mu}&a}$. Let us consider its composition with the unit of the adjunction $\eta_a\cdot\bar{\mu}$. Since $\eta_a$ is $F$-vertical, $\eta_a\cdot\bar{\mu}$ still lifts $\gamma$, and $I$ full and faithful implies the existence of a unique $\xymatrix{b\ar[r]^{\mu}&Ra}$ such that $I\mu=\eta_a\cdot\bar{\mu}$. Let us prove that $\mu$ is $FI$-opcartesian.
To this end, we consider and arrow $\beta$ of $\cB$ with domain $b$ and an arrow $\gamma'$ of $\cC$ such that $FI(\beta)=\gamma'\cdot \gamma$:
$$
\xymatrix@C=10ex{
&&b'\\
b\ar@/^3ex/[urr]^{\beta}\ar[r]^{\mu}
&Ra\ar@{-->}[ur]
&&\cB\ar[d]^{FI}
\\
FIb\ar[r]_-{\gamma}
&c\ar[r]_-{\gamma'}
&c'
&\cC
}
$$
If we apply the functor $I$ to $\beta$ and $\mu$ above, we obtain the following diagram of solid arrows in $\cA$ (over $\cC$):
$$
\xymatrix@C=10ex{
&&&Ib'\\
Ib\ar@/^3ex/[urrr]^{I\beta}\ar[r]^{\bar\mu}
&a\ar@{-->}@/^1ex/[urr]^{\alpha}\ar[r]^{\eta_a}
&IRa\ar@{-->}[ur]_{I\beta'}
&&\cA\ar[d]^{F}
\\
FIb\ar[r]_-{\gamma}
&c\ar@{=}[r]
&c\ar[r]_-{\gamma'}
&c'
&\cC
}
$$
Now, since $\bar\mu$ is $F$-opcartesian, there exists a unique $\alpha$ as above, such that $F\alpha=\gamma'$ and $\alpha\cdot\bar{\mu} = I\beta$. By the universal property of the unit of the adjunction, there exists a unique $\beta'$ in $\cB$ such that $I\beta'\cdot\eta_a=\alpha$. Then, $I(\beta'\cdot \mu)=I\beta'\cdot I\mu=I\beta'\cdot \eta_A\cdot \bar{\mu}=\alpha\cdot \bar{\mu}=I\beta$. Moreover, since $\eta$ vertical, $FI\beta'=FI\beta'\cdot F\eta_a=F(I\beta'\cdot \eta_a)=F\alpha=\gamma'$, and this concludes the proof.
\end{proof}
\begin{proof}[Proof of (v)]
Given an $F$-opcartesian arrow $\xymatrix{a\ar[r]^-{\mu}&a'}$, we are to prove that $R\mu$ is $FI$-opcartesian over $F\mu$.
First of all, we check that $R\mu$ lifts $F\mu$. Indeed, by the naturality of units, $IR\mu\cdot\eta_a=\eta_{a'}\cdot \mu$, and since the units are $F$-vertical, this implies $FIR\mu=F\mu$. Now, let us consider the situation described by the following diagram of solid arrows:
$$
\xymatrix@C=10ex{
&&b\\
Ra\ar@/^3ex/[urr]^{\beta}\ar[r]^{R\mu}
&Ra'\ar@{-->}[ur]
&&\cB\ar[d]^{FI}
\\
Fa\ar[r]_-{F\mu}
&Fa'\ar[r]_-{\gamma}
&c'
&\cC
}
$$
where $\beta$ is such that $FI\beta=\gamma\cdot F\mu$. If we apply the functor $I$
to the upper part of the  diagram above, we can consider the following commutative diagram of solid arrows in $\cA$:
$$
\xymatrix@C=10ex{
&&Ib\\
IRa\ar@/^3ex/[urr]^{I\beta}\ar[r]^{IR\mu}
&IRa'\ar@{-->}[ur]_{I\beta'}
\\
a\ar[r]_-{\mu}\ar[u]^{\eta_a}
&a'\ar[u]_-{\eta_{a'}}\ar@{-->}@/_3ex/[uur]_{\alpha}
}
$$
Since $\mu$ is opcartesian, there exists a unique $\alpha$ as above such that $\alpha\cdot\mu=I\beta\cdot \eta_a$ and $F\alpha=\gamma$. Moreover,  since $\eta_{a'}$ is a unit, there exists a unique $\beta'$ such that $I\beta'\cdot \eta_{a'}=\alpha$.
Now we claim that $I\beta'\cdot IR\mu= I\beta$. Indeed, $I\beta'\cdot\eta_a=\alpha\cdot\mu=I\beta'\cdot \eta_{a'}\cdot \mu=I\beta'\cdot IR\mu\cdot \eta_a$, and $\eta_a$ is an epimorphism.
\end{proof}

Previous propositions have the following corollary.
\begin{Corollary}\label{coroll:bifib_reflection}
Let us consider diagram (\ref{diag:basic_sit}), where $(R\dashv I, \eta,\epsilon)$ is a full replete  epi-reflection,  $F$ is an bifibration and $\eta$ has $F$-vertical components. Then the following hold:
\item[(vi)] $FI$ is an bifibration with the same cartesian lifts as $F$ and $FI$-opcartesian lifts obtained by reflecting $F$-opcartesian lifts;
\item[(vii)] $I$ is a cartesian functor over $\cC$, $R$ is an opcartesian functor over $\cC$.
\end{Corollary}

\section{Some (op)fibrational properties of reflexive \\ graphs and reflexive relations}\label{sec:Reflexive graphs and relations}

In this section, we describe how the category of internal reflexive relations  can be seen a regular epi-reflective subcategory of the category of internal reflexive graphs.
With the help of Corollary \ref{coroll:bifib_reflection}, this will allow us to deduce from the well-known bifibration of objects of internal graphs, a new bifibration of objects of internal reflexive relations.

\subsection{Reflexive graphs and reflexive relations}\label{ssec:reflexive_rel}

First we recall the definition of the  functor
$$
U\colon\xymatrix{\RGph(\cC)\ar[r]& \cC}
$$
given by the assignment $\bA=(A_1,A_0,d,c,e)\mapsto A_0$.
Then, if $A_0$ is an object of  $\cC$, we denote by  $\RGph_{A_0}(\cC)$ the fibre of $U$ over $A_0$, i.e.\ the subcategory of  $\RGph(\cC)$
of reflexive graphs with fixed object of vertexes $A_0$, and morphisms inducing identity on $A_0$.

\begin{nsc}
When $\cC$ has finite limits, the functor $U$ is a fibration, with cartesian maps given by \emph{fully faithful morphisms of graphs},
i.e.\ those $\xymatrix{F \colon \bA \ar[r] &\bB}$ such that the following diagram is a pullback:
\begin{equation}\label{diag:RGph_fully_faith}
\begin{aligned}
\xymatrix@C=10ex{
A_1\ar[d]_{\langle d,c\rangle}\ar[r]^{f_1}
&B_1\ar[d]^{\langle d,c\rangle}
\\
A_0\times A_0 \ar[r]_{f_0\times f_0}
&B_0\times B_0
}
\end{aligned}
\end{equation}
Consequently, for $\xymatrix{f\colon A_0\ar[r] &B_0,}$ a change of base functor
$$
\xymatrix{f^*\colon \RGph_{B_0}(\cC)\ar[r]&\RGph_{A_0}(\cC)}
$$
is defined by taking the joint pullback of $d$ and $c$ along $f$:
\begin{equation}\label{diag:RGph_cart}
\begin{aligned}
\xymatrix@C=12ex{\ar@{}[dr]|(.3){\lrcorner}
f^*(B_1)\ar[d]_{\langle d,c\rangle}\ar[r]^{\tilde f}
&B_1\ar[d]^{\langle d,c\rangle}
\\
A_0\times A_0 \ar[r]_{f\times f}
&B_0\times B_0
}
\end{aligned}
\end{equation}

\end{nsc}

\begin{nsc}
When $\cC$ has pushouts of split monomorphisms along arbitrary morphisms,  the functor $U$ is an opfibration, with opcartesian maps given by diagrams
as (\ref{diag:RGph_morphism}), with $$e\cdot f_0 = f_1 \cdot e$$ a pushout.

Consequently, for $\xymatrix{f\colon A_0\ar[r] &B_0,}$ a change of base functor
$$
\xymatrix{f_!\colon \RGph_{A_0}(\cC)\ar[r]&\RGph_{B_0}(\cC)}
$$
is defined by the following construction using the universal property of the pushout:
\begin{equation}\label{diag:RGph_opcart}
\begin{aligned}
\xymatrix@C=10ex{
A_0\ar[r]^{f}
&B_0
\\
A_1\ar[r]^{\hat{f}}\ar@<+.6ex>[u]^{d}\ar@<-.6ex>[u]_{c}
&f_!(A_1) \ar@<+.6ex>@{.>}[u]^{d}\ar@<-.6ex>@{.>}[u]_{c}
\\
A_0 \ar[r]_{f}\ar[u]^{e}\ar@{}[ur]|(.7){\llcorner}
&B_0\ar[u]_{e}
}
\end{aligned}
\end{equation}
where we have used the fact that $e$ is a common section of both $d$ and $c$.
\end{nsc}

\begin{nsc}
Notice that if $\cC$ has pullbacks and pushouts of split monomorphisms, $U$ is at the same time a fibration and an opfibration, \emph{i.e.}\ it is a \emph{bifibration}. As a consequence,  every
$\xymatrix{f\colon A_0\ar[r] &B_0}$ determines an adjoint pair
$$
\xymatrix@C=10ex{\RGph_{B_0}(\cC)\ar@<-1ex>[r]_{f^*}\ar@{}[r]|{\bot}
&\RGph_{A_0}(\cC)\ar@<-1ex>[l]_{f_!}\,.
}$$
\end{nsc}

\begin{nsc}
If the category $\cC$ is regular, then the functor $I$ has a left adjoint regular epi reflection
$$
\xymatrix{
R\colon \RGph(\cC)\ar[r]&\RRel(\cC)\,,
}
$$
where, for a reflexive graph $(A_1,A_0,d,c,e)$, the reflector is given by the regular epi-mono factorization of the map $\langle d,c \rangle$:
$$
\xymatrix{
A_1\ar@{->>}[r]^-{\eta_{A}}\ar[d]_{\langle d,c \rangle}
&R(A_1)\ar[dl]
\\
A_0\times A_0
}
$$
This construction specializes a result by Xarez for internal graphs and relations, see \cite{Xa04}.
\end{nsc}

\begin{nsc}\label{nsc:main_example}
Notice that the components of the unit of the adjunction $R\dashv I$ are vertical with respect to $U$, hence we can apply
 Proposition \ref{prop:fib_reflection} and Proposition \ref{prop:opfib_reflection}. The situation is described by the diagram below
$$
\xymatrix{
\RGph(\cC)\ar[dr]_{U}\ar@<+1ex>[rr]^{R}\ar@{}[rr]|{\bot}
&&\RRel(\cC)\ar@{-->}[dl]^{V=U\cdot I}\ar@<+1ex>[ll]^{I}
\\
&\cC
}
$$
Let us recall that, with any fixed object $X$ of $\cC$, it is possible to associate the discrete relation
$\Delta X=(X,1_X,1_X)$ and the codiscrete relation $\nabla X=(X\times X,p_1,p_2)$. They are
the initial and the final objects in the fibre $\RRel_X(\cC)$, as well as in the fibre $\RGph_X(\cC)$.
\end{nsc}
\begin{nsc}
Concerning the fibrations, Proposition \ref{prop:fib_reflection} and Proposition \ref{prop:opfib_reflection} recover two well known facts:
\begin{itemize}
\item $V$ has the same cartesian lifts as $U$, i.e.\ fully faithful morphisms of reflexive relations;
\item a fully faithful morphism of reflexive graphs with codomain a reflexive relation has domain a reflexive relation.
\end{itemize}
\end{nsc}

\begin{nsc}
Concerning the opfibrations, for the reader's convenience we describe the construction of opcartesian lifts along $V$.

To this end, let us consider an arrow $\xymatrix{X\ar[r]^f&Y}$ of $\cC$, and a relation $(S,s_1,s_2)$ on $X$.
The opcartesian lift of $f$  relative to $S$ with respect to $V$ is the morphism of reflective relations $(\eta\cdot \hat{f},f)$.
Its construction is displayed in the following diagram:
\begin{equation}\label{diag:RRel_opcart}
\begin{aligned}\xymatrix{
S\ar@<-.5ex>[d]_{r_1}\ar@<+.5ex>[d]^{r_2}\ar[r]^-{\hat{f}}
&f_{!}(S)\ar@<-.5ex>[d]\ar@<+.5ex>[d]\ar@{->>}[r]^-{\eta}
&R(f_{!}(S))\ar@<-.5ex>[dl]\ar@<+.5ex>[dl]
\\
X\ar[r]_-{f}&Y
}\end{aligned}
\end{equation}
where $(\hat{f},f)$ is the opcartesian lift of $f$ in $\RGph(\cC)$ and $\eta=\eta_{f_{!}(S)}$ is  the unit of the reflection.

\end{nsc}
\begin{nsc}\label{nsc:r_cartesian}
The situation described in \ref{nsc:main_example} is better behaved than the formal setting of Proposition \ref{prop:fib_reflection}, since  in \ref{nsc:main_example} both $I$ and $R$ are cartesian functors. This assertion is a straightforward consequence of the pull-back stability of the regular-epi/mono factorization in a regular category. A consequence of this fact is that the adjunction $R\dashv I$ actually lives in $\text{Fib}(\cC)$. Using the terminology of \cite{Borceux94.2}, it is called a \emph{fibred adjunction}.
However, this extra feature of $R$ will not be used in our work, since the constructions we need rely only on the formal setting developed in the previous section.
\end{nsc}

\subsection{An equivalence relations associated with a Bourn-normal mono\-morphism}
The previous discussion on the (co)fibrational aspects of reflexive relations (Section \ref{ssec:reflexive_rel}) indicates  a way for associating an equivalence relation with a given Bourn-normal monomorphism.

\begin{Proposition}\label{prop:Rel(n)}
Let $\cC$ be a Mal'tsev regular category, with pushout of split monomorphisms. Then the assignment
$$\Rel(n)=R(n_!(\nabla N))$$
  defines a  functor
$$
\xymatrix{\Rel\colon\mathsf{Arr}(\cC)\ar[r]& \EqRel(\cC)}
$$
that associates with the arrow $\xymatrix{N\ar[r]^{n}&X}$ of $\cC$, the equivalence relation $\Rel(n)$ on $X$.

If restricted to Bourn-normal monomorphisms,   this functor
is faithful, and $\Rel(n)$ is an equivalence relation to which $n$ is Bourn-normal.
More precisely,  $\Rel(n)$ is the initial equivalence relation among the equivalence relations to which  $n$ is Bourn-normal.
\end{Proposition}

\begin{proof}
Since we defined $\Rel(n)$ as the codomain of the opcartesian lift of $n$ relative to $\nabla N$ (with respect to the functor $V$ of \ref{nsc:main_example}), it obviously extends to a functor $\xymatrix{\mathsf{Arr}(\cC)\ar[r]& \EqRel(\cC).}$ Furthermore, when restricted to monomorphisms, this functor is clearly faithful.

It remains to check is that when $n$ is a Bourn-normal monomorphism, then it is Bourn-normal to $\Rel(n)$. Let  $S$ be a relation to which $n$ is Bourn-normal, together with the discrete fibration $\xymatrix{m=(m,n)\colon\nabla N\ar[r]&S.}$ Let us stress that, although we do not know $S$ and $m$ \emph{a priori}, we know that one such a pair $(S, m)$ exists in $\EqRel(\cC)$,  $n$ \emph{being} Bourn-normal.
Since the  discrete fibration $m$ is a morphism over $n$, there is a univocally determined factorization:
$$
\xymatrix{
m\colon \nabla N\ar[r]^-{\hat{n}}
&R(n_!(\nabla N))\ar[r]^-{\rho}
&S
}
$$
through the $V$-opcartesian lift $\hat{n}$ of $n$.
We have to prove that the $\hat{n}$ is a discrete fibration, but, since $\rho$ is mono, this follows by the Lemma \ref{lemma:elementary} below.

Finally, we observe that the monomorphism $\xymatrix{\rho\colon \Rel(n)\ar[r]&S}$ together with the arbitrary choice of $S$ shows that $\Rel(n)$ is initial.
\begin{Lemma}\label{lemma:elementary}
Let us consider the following diagram in the category $\cC$:
$$
\xymatrix{
\bullet\ar[r]\ar[d]\ar@{}[dr]|{(i)}
&\ar[r]^u\ar[d]_v\ar@{}[dr]|{(ii)}
&\bullet\ar[d]
\\
\bullet\ar[r]
&\bullet\ar[r]_{v}
&\bullet
}
$$
with $u$ and $v$ jointly monic. Then, if $(i)+(ii)$ is a pullback, also
$(i)$ is a pullback.
\end{Lemma}
\end{proof}
In \cite{Bo00b} Bourn proved that, in the protomodular case, if a monomorphism is Bourn-normal to an equivalence relation, then such a relation is essentially unique. Therefore it is immediate to state the following corollary.
\begin{Corollary}
If the base category $\cC$ is not only Mal'tsev, but protomodular, then, for a Bourn-normal monomorphism $n$ as above, $\Rel(n)$ is the essentially unique equivalence relation associated with $n$. In other words, the monomorphism $\rho$ is actually an isomorphism.
\end{Corollary}
\begin{Remark}
The functor $R$ defined above can be considered as a \emph{direct image} functor $\cC/X\to \EqRel_X(\cC)$ in the sense of \cite{Law70}. It can be deduced from Theorem \ref{thm:The_Theorem} that such a functor admits a right adjoint if and only if we restrict it to Bourn-normal monomorphisms with  null support. However, in this case, the unit of the adjunction is an isomorphism.
\end{Remark}
\begin{Example}\label{expl:Rel_in_Gp}
It is worth examining the situation in the semi-abelian (hence protomodular) category of groups. Hence, let us consider a  group $X$ together with a normal inclusion $n\colon N\hookrightarrow X$ of the normal subgroup $N$. We  compute the amalgamated product
$$
n_!(\nabla N) = (N\times N) *_{N} X
$$
From elementary group theory, we know that this amounts to the free product of $(N\times N)$ and $X$ over (its normal subgroup generated by) the relations $(n,n);n^{-1}=1$, for all $n\in N$.
The second step is the construction of  $R(n_!(\nabla N))$. This can be done by taking the image of the group homomorphism
$$
n_!(\nabla N)\to X\times X
$$
generated by the assignments
$$
(n,n')\mapsto (n,n')\,,\qquad x\mapsto (x,x)\,.
$$
For instance, the string
$$(n_1,n_1');x_1;(n_2,n_2');x_2;\,\cdots\,;(n_k,n_k');x_k$$
will be sent to the pair
\begin{equation}\label{eq:Rel_in_Gp}
(n_1 x_1  n_2 x_2  \cdots  n_k x_k \, ,\,\, n_1' x_1  n_2' x_2  \cdots  n_k' x_k)
\end{equation}
Now,  since $N$ is normal in $X$,
\begin{eqnarray*}
n_1 x_1  n_2 x_2 n_3 x_3 \cdots  n_k x_k &=&
x_1 x_1^{-1} n_1 x_1  n_2 x_2 n_3 x_3 \cdots  n_k x_k \\
&=& x_1\bar{n}_2 x_2 n_3 x_3 \cdots  n_k x_k \\
&=& x_1 x_2 x_2^{-1}  \bar{n}_2 x_2 n_3 x_3 \cdots  n_k x_k \\
&=& x_1 x_2 \bar{n}_3 x_3  \cdots  n_k x_k  \\
&& \cdots\\
&=& x_1 x_2 x_3 \cdots x_{k-1} \bar{n}_k x_k \\
&=& x_1 x_2 x_3 \cdots x_k \bar{n}= \bar{x}\bar{n}
\end{eqnarray*}
so that the pair (\ref{eq:Rel_in_Gp}) can be written in the form $(\bar{x}\bar{n},\bar{x}\bar{n}')$, and both $\bar{x}\bar{n}$ and $\bar{x}\bar{n}'$ lie in the same coset $\bar{x}N$ of the normal subgroup $N$ of $X$.

Indeed, it is easy to show that every element of $(N\times N) *_{N} X$ can be expressed in this form, and that the condition is also sufficient;  equivalently, that the canonical monomorphism
$$
R(n_!(\nabla N))\to R_N
$$
is an isomorphism, where $R_N$ is the kernel pair relation of the canonical projection $X\to X/N$.
\end{Example}

\begin{nsc}
A relevant consequence of Proposition \ref{prop:Rel(n)}, is that, under the hypotheses of the proposition, one can remove the existential quantifier from the definition of normal monomorphism. This is achieved by the following statement.
\end{nsc}
\begin{Corollary}\label{cor:no_exists}
Let $\cC$ be a Mal'tsev regular category, with pushout of split monomorphisms. An arrow
$\xymatrix{N \ar[r]^n &X}$ is Bourn-normal \emph{(\emph{to} $\Rel(n)$)} if and only if the morphism
$$
\xymatrix{
\nabla N\ar[r]^-{\hat{n}}
&R(n_!(\nabla N))
}
$$
is a cartesian discrete fibration in $\EqRel(\cC)$.
\end{Corollary}

\begin{nsc}
We shall denote by $\EqRel_{\,0}(\cC)$ the essential image of the functor
$\Rel$, i.e.\ the full subcategory of $\EqRel(\cC)$ determined by those equivalence relations $S$ with $S\simeq\Rel(n)$, for some Bourn-normal monomorphism $n$.
We get a factorization $\Rel=L'\cdot \Rel_0$ through the embedding
$$
\xymatrix{
L'\colon \EqRel_{\,0}(\cC)\ \ar@{>->}[r]&\EqRel(\cC)
}
$$
\end{nsc}

\begin{Remark}
As it is clear from the construction in the proof of Proposition \ref{prop:Rel(n)}, the hypothesis that $\cC$ is a Mal'tsev category is considered only in order to collapse $\EqRel(\cC)$ with $\RRel(\cC)$.

We can actually remove this hypothesis, and let the functor $\Rel$ land in $\RRel(\cC)$. In this case, the notion of Bourn-normal subobject would be replaced by a new concept that, in the pointed regular case, is related with the notion of \emph{clot}, i.e.\ the normalization of an internal reflexive relation. For an algebra $X$ of a pointed (i.e. with only one constant) variety of universal algebras, and a subalgebra $(N,n)$ of $X$, $\Rel(n)$ is the subalgebra of $X\times X$ generated by $\mathsf{diag}(X)\cup (N\times N)$, where $\mathsf{diag}(X)$ is the diagonal of $X\times X$.

The reader may consult \cite{MM10a} (and the references therein) for a discussion on the different aspects of normality in pointed contexts.
\end{Remark}

\section{The big picture}

The discussion developed in the last section establishes, for a quasi-pointed regular Mal'tsev category $\cC$ with pushout of split monomorphisms along arbitrary morphisms, a pair of functors relating in both directions the category of internal equivalence relations $\EqRel(\cC)$ with the category of Bourn-normal monomorphisms $\mathsf{N}(\cC)$:
\begin{equation}\label{diag:not_so_an_adjunction}
\xymatrix@C=10ex{
\mathsf{N}(\cC)\ar@<-.6ex>[r]_-{\Rel}
&\EqRel(\cC)\ar@<-.6ex>[l]_-{\Nor}
}
\end{equation}

\begin{nsc}
Things thus standing, it is natural to investigate whether this pair of functors are related by an adjunction. Indeed, Corollary \ref{cor:norm_0-norm} would suggest to define a monomorphic natural transformation with components
$$
\xymatrix{\epsilon_n\colon \Nor(\Rel(n))\ \ar@{>->}[r]& n}
$$
for $n$ in $\mathsf{N}(\cC)$. On the other hand, by Proposition \ref{prop:Rel(n)}, one would define \emph{another} monomorphic natural transformation with components
$$
\xymatrix{\epsilon'_S\colon \Rel(\Nor(S))\ \ar@{>->}[r]& S}
$$
for $S$ in $\EqRel(\cC)$.

Both these transformations would be eligible as counits of an adjunction, but the directions do not match, since triangular identities are replaced by the two equations of the following lemma.
\end{nsc}

\begin{Lemma}\label{lemma:not_so_triangular}
Let $\cC$ be a quasi-pointed regular Mal'tsev category with pushout of split monomorphisms along arbitrary morphisms. The natural transformations $\epsilon$ and $\epsilon'$ defined above satisfy the following two equations:
\begin{equation}\label{eq:not_so_triangular}
\Rel(\epsilon_n)=\epsilon'_{\Rel(n)}\, , \qquad \Nor(\epsilon'_S)=\epsilon_{\Nor(S)}\, .
\end{equation}
Moreover, $\Nor(\epsilon'_S)$ is an isomorphism.
\end{Lemma}
\begin{proof}
Concerning the first equality, let us consider a Bourn-normal monomorphism $\xymatrix{N\ar[r]^n&X}$.  By Proposition \ref{prop:norm_0-norm}, $\epsilon_n$ is given by the kernel $k$ of the final map $\xymatrix{N\ar[r]&1}$:
$$
\epsilon_n\colon\begin{aligned}
\xymatrix{K\ar[r]^k\ar[dr]_{\Nor(\Rel(n))}&N\ar[d]^{n}\\&X}
\end{aligned}
$$
Then we compute $\Rel(\epsilon_n)$: its construction is represented by the right hand face in the following diagram:
$$
\xymatrix{
K\times K\ar[rr]^{\widehat{nk}} \ar@<-.5ex>[d]\ar@<+.5ex>[d]\ar[ddr]^(.25){k\times k}
&&R((n\!\cdot\!k)_{!}(\nabla K)) \ar@<-.5ex>[d]\ar@<+.5ex>[d]\ar[ddr]^{\Rel(\epsilon_n)}
\\
K\ar[rr]^{n\cdot k}|(.25)\hole \ar[ddr]_{k}
&&X\ar@{=}[ddr]|(.5)\hole
\\
&N\times N\ar[rr]^{\widehat{n}} \ar@<-.5ex>[d]\ar@<+.5ex>[d]
&&R(n_{!}(\nabla N)) \ar@<-.5ex>[d]\ar@<+.5ex>[d]
\\
&N\ar[rr]_{n}
&&X
}
$$
However, $\Rel(n)$ is given by the relation $R(n_{!}(\nabla N)) $ represented by the front face in the above diagram, so that the monomorphic comparison $\rho$ of Proposition \ref{prop:Rel(n)} (that defines $\epsilon'$) is precisely the right hand face in the diagram, and this concludes the proof of the first equality.

Concerning the second equality, it can be derived from the following facts: (a) for any Bourn-normal monomorphisms $n$ and $n'$ on the same object $X$, there exists at most one morphism in $\mathsf{N}(\cC)$ that gives the identity on $X$; (b) both $\Nor(\epsilon'_S)$ and $\epsilon_{\Nor(S)}$ determine such a morphism between Bourn-normal monomorphisms.
Finally, a direct computation shows that $\Nor(\epsilon'_S)=\epsilon_{\Nor(S)}$ is an isomorphism.

\end{proof}
The next theorem follows from Lemma \ref{lemma:not_so_triangular}, and it is the main result of this section.
\begin{Theorem}\label{thm:The_Theorem}
Let $\cC$ be a quasi-pointed regular Mal'tsev category with pushout of split monomorphisms along arbitrary morphisms, and let us consider the pair of functors $(\Rel,\Nor)$ as in diagram \emph{(\ref{diag:not_so_an_adjunction})}. Then
\begin{itemize}
\item the restriction to the essential image $\mathsf{N}_0(\cC)$ yields the monomorphic coreflection:
\begin{equation}\label{diag:the_adjunctions}
\xymatrix@C=8ex{
\mathsf{N}_0(\cC)\ar@<-1ex>[r]_-{\Rel}\ar@{}[r]|-{\top}
&\EqRel(\cC)\ar@<-1ex>[l]_-{\Nor}
}
\end{equation}
\item the further restriction to the essential image $\Rel(\mathsf{N}_0(\cC))$ yields the adjoint equivalence
\begin{equation}\label{diag:one_equivalence}
\xymatrix@C=10ex{
\mathsf{N}_0(\cC)\ar@<-1ex>[r]_-{\Rel}\ar@{}[r]|-{\simeq}
&\Rel(\mathsf{N}_0(\cC))\ar@<-1ex>[l]_-{\Nor}
}
\end{equation}

\end{itemize}
\end{Theorem}
\begin{proof}
Using Proposition \ref{prop:ker}, one easily checks that $\epsilon_n=(k,1_X)$, where $k$ is given by the pullback:
$$
\xymatrix{
K\ar[d] \ar[r]^{k}\ar@{}[dr]|(.4){\lrcorner}
&N\ar[d]\ar[r]^n&X
\\0\ar[r]&1
}
$$
When $n\in \mathsf{N}_0(\cC)$, the pullback square above can be decomposed as shown below:
$$
\xymatrix{
K\ar[d] \ar[r]^{k}
&N\ar[d]
\\0\ar[r]\ar[d] \ar[r]&0\ar[d]
\\0\ar[r]&1
}
$$
and since the lower region of the diagram and the whole is a pullback, also the upper region is, so that $k$ is an isomorphism.
As a consequence, one can define $\eta'=\epsilon^{-1}$, so that the pair $(\eta',\epsilon')$ gives the unit and the counit of the adjunction $\Rel \dashv \Nor$ of the first statement of the proposition. Then one easily prove that the faithful functor $\Rel$ is also full.

The second statement amounts to the obvious fact that any fully faithful functor gives an equivalence when restricted to its essential image.
\end{proof}

\begin{nsc}
The following diagram summarizes the situation in the quasi-pointed Mal'tsev regular setting:
\end{nsc}
\begin{equation}\label{diag:the_big_picture}
\begin{aligned}
\xymatrix@=20ex{
\mathsf{N}_0(\cC)\ar@<-1ex>[r]_-{\Rel}\ar@{}[r]|-{\simeq}\ar[d]_{L}
\ar@<-1ex>[dr]_(.57){\Rel}\ar@{}[dr]|{\top}
&\Rel(\mathsf{N}_0(\cC))\ar@<-1ex>[l]_-{\Nor}\ar[d]^{L'}
\\
\mathsf{N}(\cC)\ar@<-.8ex>[r]_-{\Rel}
&\EqRel(\cC)\ar@<-.8ex>[l]_-{\Nor}
\ar@<-1ex>[ul]_(.57){\Nor}}
\end{aligned}
\end{equation}
The pair of functors $\Nor$ and $\Rel$ gives a comparison (in the sense of Lemma \ref{lemma:not_so_triangular}) between the category of Bourn-normal monomorphisms and that of equivalence relations. The inclusions $L$ and $L'$ restrict this comparison to an equivalence of categories.

In fact, the functor $L$ is related with the possible pointedness of $\cC$, while $L'$ with its protomodularity. More precisely, one can easily see that if $\cC$ is pointed, $L$ is the identity, and that if $\cC$ is protomodular,  $L'$ is the identity. Consequently, when $\cC$ is a pointed protomodular regular category, the equivalence above reduces to the one already studied by Bourn in \cite{Bo00b}.
Actually, in this case, it restricts further to the well known equivalence between normal monomorphisms and  effective equivalence relations (see also \cite{MM10a}). However the role of effective equivalence relations and normal monomorphisms is somehow more obscure in the more general quasi-pointed  Mal'tsev case.

\begin{Remark}
It is not difficult to show that the embedding
$$
\xymatrix{L\colon \mathsf{N}_0(\cC)\ \ar@{>->}[r] & \mathsf{N}(\cC)}
$$
presents $\mathsf{N}_0(\cC)$ as a mono-coreflective subcategory of $\mathsf{N}(\cC)$, with coreflection given by $J=\Nor\cdot\Rel$.
However, the functor $J$ could have been defined directly as follows. Consider the construction given below for the domain $N$ of a Bourn-normal monomorphism $(N,n)$:
\begin{equation}\label{diag:norm_to_0-norm}
\begin{aligned}
\xymatrix{
K\times K\ar[r]^-{k\times k}\ar@<-.5ex>[d]_{p_1}\ar@<+.5ex>[d]^{p_2}
&N\times N\ar@<-.5ex>[d]_{p_1}\ar@<+.5ex>[d]^{p_2}
\\
K\ar[r]_-{k}\ar[d]
&N\ar[d]
\\
0\ar[r]
&1
}
\end{aligned}
\end{equation}
where $k=\ker(N\to 1)$, and the two vertical forks are left exact, i.e.\ kernel pairs.
Therefore, the upper squares are pullbacks. By pasting those with the right hand diagram of
(\ref{diag:normal}), one gets that $n\cdot k$ is Bourn-normal (to the same equivalence relation(s) to which also
$n$ is Bourn-normal). Then, we recover $J(n)=n\cdot k$.

In this way, one can appreciate the adjunction $L\dashv J$ as a consequence of a more general adjunction $L''\dashv J''$
$$
\xymatrix@C=10ex{
\mathsf{N}_0(\cC)\ar@<-1ex>[d]_{L}\ar@{}[d]|-{\dashv}\ar[r]^{\text{dom}}
&\cC_0\ar@<-1ex>[d]_{L''}\ar@{}[d]|-{\dashv}
\\
\mathsf{N}(\cC)\ar[r]_{\text{dom}}\ar@<-1ex>[u]_{J}
&\cC\ar@<-1ex>[u]_{J''}}$$
where $J''(X)=\text{Ker}(X\to 1)$.
\end{Remark}

\section{Examples}
The following examples will help to clarify the possible scenarios that arise from the theory developed so far.

\begin{Example}\label{expl:Gp}
A first example is given by the category $\Gp$ of groups.
This category is semi-abelian, hence Barr-exact, pointed and protomodular. In $\Gp$,  all equivalence relations are effective, hence, by Proposition 10 in \cite{Bo00b}, all Bourn-normal monomorphisms are normal. In other words, diagram (\ref{diag:the_big_picture}) above reduces to the well known equivalence between normal monomorphisms and internal equivalence relations in groups (see Example \ref{expl:Rel_in_Gp} for the construction of the functor $\Rel$ in $\Gp$).
\end{Example}

Actually, the previous example is not very enlightening with respect to the general case, since relevant items that are in general distinguished  collapse in semi-abelian categories. However it has been included here since we shall present different generalizations of it in the examples that follow.

\begin{Example}\label{expl:GpTop}

In the second example, we aim to remove the Barr-exactness condition. Hence, we present the case $\Gp(\Top)$ of topological groups, since, as it has been proved in \cite{BC05}, $\Gp(\Top)$ is still homological, i.e.\  a (finitely complete) pointed, protomodular, regular category.

Likewise in the semi-abelian case, $L$ and $L'$ are identities, so that the pair $(\Nor, \Rel)$ establishes an adjoint equivalence
$$
\mathsf{N}(\Gp)\simeq\EqRel(\Gp)
$$
between Bourn-normal monomorphisms and internal equivalence relations. Moreover, such an equivalence  restricts to normal monomorphisms and effective equivalence relations:
$$
\mathsf{K}(\Gp)\simeq\EffRel(\Gp)\,.
$$
Indeed, in the category of topological groups, a Bourn-normal subobject
$$\xymatrix{N\ar[r]^{n}&X}$$
is normal precisely when it is an algebraic kernel (i.e.\ a normal monomorphism in $\Gp$) endowed with the induced topology. Then one can see that, given a Bourn-normal subobject $n:N\to X$, the equivalence relation $\Rel(n)$ is effective precisely when $n$ is normal.
\end{Example}

\begin{Example}\label{expl:GpdX}
Now we analyze a situation that, with respect to the case of groups, keeps the Barr-exactness but not the pointedness. To this end, let us consider $\Gpd_S$, i.e.\ the category of groupoids over a fixed set of objects $S$ and constant-on-object functors between them. In other words, this is the fibre over $S$ of the \emph{functor of objects} $\Gpd\to \Set$.

The case of groups is recovered by taking as $S$ the final object:
$$\Gpd_1\simeq \Gp\,,$$
However, although $\Gpd_1$ is a pointed category, $\Gpd_S$ is in general only quasi-pointed, with initial maps given by the inclusion of the set $S$ of objects, considered as a discrete groupoid:
$$
\Delta S\to \bX\,.
$$
Notice that in $\Gpd_S$,  objects with null support are precisely those endowed with a (necessarily unique) constant-on-objects functor to the discrete groupoid $\Delta S$, i.e.\ the totally disconnected groupoids over $S$.

The Bourn-normal subobjects in $\Gpd_S$ are characterized in \cite{Bo02}: given a groupoid $\bX$ as above, a subobject $n\colon \bN\to \bX$ (in $\Gpd_S$) is Bourn-normal if for every oarrow $\alpha\colon y\to y$ in $\bN$ and every arrow $f\colon x\to y$, the \emph{conjugate} $f^{-1}\alpha f$ is in $\bN$.
Then, by Proposition \ref{prop:ker}, normal monomorphisms are precisely the totally disconnected Bourn-normal subobjects. Moreover, given a Bourn-normal subobject $(\bN,n)$, its associated normal subobject $\Nor(\Rel(n))$ is the kernel of $\bN\to \nabla S$, i.e.\ the maximal disconnected subgroupoid of $\bN$.

The category $\Gpd_S$ is protomodular (see \cite{Bo91} for a proof), so that the inclusion $L'$ of diagram (\ref{diag:the_big_picture}) is an identity, and the pair $(\Nor,\Rel)$ establishes an equivalence of categories $\mathsf{N}_0(\cC)\simeq \EqRel(\cC)$. In fact, by Barr-exactness, this is an equivalence between normal monomorphisms and effective equivalence relations.

In conclusion, for this quasi-pointed protomodular case, we observe that given a Bourn-normal subobject, it is Bourn-normal to a unique equivalence relation, while on the other side, a given equivalence relation may have several Bourn-normal subobjects associated with it, but only one with null support, which is contained in every other subobject which is Bourn-normal to the equivalence relation.
\end{Example}

I am indebted with the anonymous referee from whom I learnt about the next example.
\begin{Example}\label{expl:Gp0} As the last example, we analyze a quasi-pointed Mal'tsev variety. Let $\Gp^{\circ}$ denote the variety of universal algebras with binary operation $x\cdot y$ and unary operation $x^{-1}$ such that $\cdot$ is associative, $x\cdot x^{-1}=y\cdot y^{-1}$, and $x\cdot x\cdot x^{-1}=x\cdot x^{-1}\cdot x=x$.
It is easy to see that a nonempty algebra $X$ of $\Gp^{\circ}$ is necessarily a group, with identity given by $x\cdot x^{-1}$, and that $\Gp^{\circ}$ is obtained from the variety of groups by adding the empty set as an object and empty maps as morphisms. Being a variety, $\Gp^{\circ}$ is certainly Barr-exact, and the inclusion map of the initial algebra (the empty set) in the final algebra (the singleton) makes it quasi-pointed. In fact, it is also Mal'tsev, since the empty relation on the empty set is clearly an equivalence relation, and all other reflexive relations are reflexive relations of groups, hence equivalence relations. However, it is  not protomodular, since pulling back along the initial maps does not reflect, in general, an isomorphism of points.

Nonempty Bourn-normal subobjects in $\Gp^{\circ}$ are precisely the normal monomorphisms in $\Gp$, but they are not normal in $\Gp^{\circ}$. In $\Gp^{\circ}$, with any equivalence relation is always associated an empty normal subobject, since only the empty set has null support. Hence, in $\Gp^{\circ}$, all nonempty equivalence relations admit precisely two Bourn-normal subobjects: their normalization in the category of groups and an empty subalgebra. In fact, fixed an algebra $X$, the only possible normal subobject lying in $X$ is the initial arrow (the empty subalgebra), and it can be obtained as the $\Nor$malization of any equivalence relation one can define on $X$. Therefore, the diagonal adjunction in diagram \ref{diag:the_big_picture} is definitely not an equivalence, in general.
Finally, $\mathsf{N}_0(\cC)$ and $\Rel(\mathsf{N}_0(\cC))$ are trivial, i.e.\ they can be identified with the base category $\cC$.

\end{Example}

Other examples can be obtained by considering internal versions of Example \ref{expl:GpdX}, where instead of $\Set$ one can consider other well-behaved categories. This is the case, for instance, of $\Gpd_S(\cE)$, for a given object $S$ of the Barr-exact category $\cE$. Similarly one can consider a regular category $\cE$ such as for instance $\Gp(\Top)$ thus obtaining topological models of internal groupoids in groups.

Also Example \ref{expl:Gp0} can be extended by considering its topological models, $\Gp^{\circ}(\Top)$. This yields a non-protomodular Mal'tsev category which is not Barr-exact, but still regular.

\section*{Acknowledgments}
I would like to thank Sandra Mantovani and Alan Cigoli for many interesting conversations and suggestions on the subject. Moreover, I would like to thank the anonymous referee whose work has been precious for correcting and improving the first version of this article. This work was partially supported by I.N.D.A.M.\ - Gruppo Nazionale per le Strutture Algebriche e Geometriche e le loro Applicazioni, and by the Fonds de la Recherche Scientifique -  F.N.R.S.: 2015/V 6/5/005 - IB/JN - 16385.


\begin{thebibliography}{}
\bibliographystyle{alpha}
\bibitem{Ba71} {\sc M. Barr,} Exact categories. Lecture Notes in Mathematics, Vol. 236, Springer, Berlin  (1971)  1--120.
\bibitem{Borceux94.2} {\sc F. Borceux,} Handbook of Categorical Algebra. Vol 2. Cambridge University Press (1994).
\bibitem{Borceux04} {\sc F. Borceux,} Non-pointed strongly protomodular theories, Applied categorical structures \textbf{12} (2004) 319--338.
\bibitem{BB} {\sc F. Borceux and D. Bourn,} Mal'cev, Protomodular, Homological and Semi-abelian Categories. Kluwer Academic Publishers (2004).
\bibitem{BC05} {\sc F. Borceux and M. M. Clementino,} Topological semi-abelian algebras. Advances in Mathematics \textbf{190} (2005) 425--453.
\bibitem{Bo91} {\sc D. Bourn,} Normalization equivalence, kernel equivalence and affine categories. Lecture Notes in Mathematics, Vol. 1488, Springer, Berlin,  1991,  43--62.
\bibitem{Bo00b} {\sc D. Bourn,} Normal subobjects and abelian objects in protomodular categories. Journal of Algebra \textbf{228} (2000) 143--164.
\bibitem{Bo01} {\sc D. Bourn,} $3\times 3$ Lemma and Protomodularity. Journal of Algebra \textbf{236} (2001) 778--795.
\bibitem{Bo02} {\sc D. Bourn,} Aspherical abelian groupoids and their directions. Journal of Pure and Applied Algebra \textbf{168} (2002) 133--146.
\bibitem{Law70}{\sc F.W. Lawvere,} Equality in hyperdoctrines and comprehension schema as an adjoint functor. \emph{In} A. Heller, ed., Proc. New York Symp. on Applications of Categorical Algebra, AMS, (1970) 1–-14.
\bibitem{MM10a} {\sc  S. Mantovani and G. Metere,} Normalities and commutators. Journal of Algebra \textbf{324} (2010), no. 9, 2568--2588.
\bibitem{Gr59} {\sc A. Grothendieck,} Cat\'egores fibr\'ees et descente. S\'eminaire Bourbaki, 1959.
\bibitem{Xa04} {\sc J. Xarez,} Internal monotone-light factorization for categories via preorders. Theory and Applications of Categories, \textbf{13} No. 15 (2004) 235--251.

\end{thebibliography}
\end{document}